\documentclass[preprint,3p,11pt,authoryear]{elsarticle}

\usepackage{amssymb,amsthm,mathtools}
\usepackage{lineno}  

\DeclareRobustCommand{\brkbinom}{\genfrac\{\}{0pt}{}}
\usepackage{enumerate,natbib,enumitem,geometry}
\usepackage{appendix} 
\usepackage{mathrsfs} 
\usepackage{dsfont} 

\usepackage[colorlinks]{hyperref}

\newtheorem{theorem}{Theorem}[section]

\newtheorem{corollary}[theorem]{Corollary}
\newtheorem{lemma}[theorem]{Lemma}
\newtheorem{remark}[theorem]{Remark}

\makeatletter \@addtoreset{equation}{section} \makeatother

\newcommand{\N}{\mathbb{N}}
\newcommand{\R}{\mathbb{R}}

\newcommand{\QQ}{\mathbb{Q}}
\newcommand{\PP}{\mathbb{P}}
\newcommand{\EE}{\mathbb{E}}

\newcommand{\VV}{\mathbb{V}\mathrm{ar}}

\newcommand{\OO}{\mathcal O}
\newcommand{\oo}{\mathrm{o}}
\newcommand{\leqdef}{\vcentcolon=}
\newcommand{\reqdef}{=\vcentcolon}
\newcommand{\rd}{{\rm d}}
\newcommand{\ind}{\mathds{1}}

\begin{document}

\begin{frontmatter}

    \title{On the Le Cam distance between Poisson and Gaussian experiments and the asymptotic properties of Szasz estimators}%

    \author[a1]{Fr\'ed\'eric Ouimet\texorpdfstring{\corref{cor1}\fnref{fn1}}{)}}%

    \address[a1]{California Institute of Technology, Pasadena, USA.}%

    \cortext[cor1]{Corresponding author}%
    \ead{ouimetfr@caltech.edu}%

    \fntext[fn1]{F.\ O.\ is supported by a postdoctoral fellowship from the NSERC (PDF) and FRQNT (B3X suppl).}%

    \begin{abstract}
        In this paper, we prove a local limit theorem for the ratio of the Poisson distribution to the Gaussian distribution with the same mean and variance, using only elementary methods (Taylor expansions and Stirling's formula).
        We then apply the result to derive an upper bound on the Le Cam distance between Poisson and Gaussian experiments, which gives a complete proof of the sketch provided in the unpublished set of lecture notes by \cite{Pollard_2010_Beijing_Chap_1}, who uses a different approach.
        We also use the local limit theorem to derive the asymptotics of the variance for Bernstein c.d.f.\ and density estimators with Poisson weights on the positive half-line (also called Szasz estimators). The propagation of errors in the literature due to the incorrect estimate in Lemma~2 (iv) of \cite{MR2960952} is addressed in the Appendix.
    \end{abstract}


\end{frontmatter}

\section{Introduction}\label{sec:intro}

    For any $\lambda > 0$, the $\mathrm{Poisson}\hspace{0.2mm}(\lambda)$ probability mass function is defined by
    \vspace{-1mm}
    \begin{equation}\label{eq:Poisson.pdf}
        P_{\lambda}(k) = \frac{\lambda^k}{k!} e^{-\lambda}, \quad k\in \N_0.
    \end{equation}
    The first objective of our paper is to derive, using only elementary methods, an asymptotic expansion for \eqref{eq:Poisson.pdf} in terms of the Gaussian density with the same mean and variance evaluated at $k$, namely:
    \begin{equation}\label{eq:phi.M}
        \frac{1}{\sqrt{2 \pi \lambda}} \exp\Big(-\frac{(k - \lambda)^2}{2 \lambda}\Big).
    \end{equation}

    \vspace{1mm}
    \noindent
    This kind of expansion can be useful in all sorts of estimation problems; we give two examples in Section~\ref{sec:applications} related to the Le Cam distance between Poisson and Gaussian experiments and the asymptotic properties of Bernstein estimators with Poisson weights (also called Szasz estimators).
    For a general presentation on local limit theorems, we refer the reader to \cite{MR1295242}.

    \newpage
    \begin{remark}
        Throughout the paper, the notation $u = \OO(v)$ means that $\limsup |u / v| < C$, as $\lambda\to \infty$, $m\to \infty$ or $n\to \infty$, depending on the context, where $C\in (0,\infty)$ is a universal constant.
        Whenever $C$ might depend on a parameter, we add a subscript (for example, $u = \OO_x(v)$).
        Similarly, $u = \oo(v)$ means that $\lim |u / v| = 0$ as $\lambda\to \infty$, $m\to \infty$ or $n\to \infty$, and subscripts indicate which parameters the convergence rate can depend on.
    \end{remark}

\section{A local limit theorem for the Poisson distribution}\label{sec:main.result}

    General local asymptotic expansions of probabilities related to the sums of lattice random variables are well-known in the literature, see e.g.\ Chapter IV in \cite{MR14626}, \cite{MR207011}, Theorem~1 in \cite{doi:10.1137/1114060}, Theorem~in \cite{MR0259992}, Theorem~1 in \cite{doi:10.1007/BF00967926}, Theorem~22.1 in \cite{MR0436272}, etc.
    However, the generality of the statements for these kinds of results often makes them difficult to apply or makes the practitioner have to rely on the validity of estimates that are difficult to verify.

    \vspace{3mm}
    For instance, based on Fourier analysis results in \cite{MR14626}, \cite[Result 4.2.7]{MR207011} shows that, uniformly in $k$,
    \begin{equation}\label{eq:LLT.Poisson.more.precise.1}
        \begin{aligned}
            P_{\lambda}(k)
            &= \frac{1}{\sqrt{\lambda}} \left\{\hspace{-1mm}
                \begin{array}{l}
                    \phi(\delta_k) - \frac{\phi^{(3)}(\delta_k)}{6 \lambda^{1/2}} + \frac{\phi^{(4)}(\delta_k)}{24 \lambda} + \frac{\phi^{(6)}(\delta_k)}{72 \lambda} \\[1mm]
                    - \frac{\phi^{(5)}(\delta_k)}{120 \lambda^{3/2}} - \frac{\phi^{(7)}(\delta_k)}{144 \lambda^{3/2}} - \frac{\phi^{(9)}(\delta_k)}{1296 \lambda^{3/2}} + \oo(\lambda^{-2})
                \end{array}
                \hspace{-1mm}\right\}, \quad \text{where } \delta_k \leqdef \frac{k - \lambda}{\sqrt{\lambda}},
        \end{aligned}
    \end{equation}
    and where $\phi^{(n)}(x) = (-1)^n \mathrm{He}_n(x) \phi(x)$, $\phi$ denotes the density function of the standard normal distribution, and $\mathrm{He}_n$ is the $n$-th probabilists' Hermite polynomial.
    If we ignore the exact form of the terms on the second line inside the braces, we get, uniformly in $k$,
    \begin{equation}\label{eq:LLT.Poisson.more.precise.1.next}
        \begin{aligned}
            P_{\lambda}(k)
            &= \frac{1}{\sqrt{\lambda}} \phi(\delta_k) \left\{1 + \frac{1}{\sqrt{\lambda}} \Big(\frac{1}{6} \delta_k^3 - \frac{1}{2} \delta_k\Big) + \frac{1}{\lambda} \Big(\frac{1}{72} \delta_k^6 - \frac{1}{6} \delta_k^4 + \frac{3}{8} \delta_k^2 - \frac{1}{12}\Big)\right\} \\[1mm]
            &\quad+ \OO\Big(\lambda^{-2} (|\delta_k| + |\delta_k|^9) \phi(\delta_k)\Big) + \oo(\lambda^{-2}).
        \end{aligned}
    \end{equation}
    To be clear, the factor $|\delta_k| + |\delta_k|^9$ in the $\OO(\cdot)$ term comes from the fact that the smallest and highest powers encountered inside the 5-th, 7-th and 9-th Hermite polynomials are $1$ and $9$, respectively (in particular, there are no constant terms in these three polynomials).
    If we want to be more precise, the $\oo(\lambda^{-2})$ error in \eqref{eq:LLT.Poisson.more.precise.1.next} has the following form (see Theorem 5 in \cite{MR14626} or Result 3.8 in \cite{MR207011} for details):
    For any fixed $s \geq 6$,
    \begin{equation}\label{eq:LLT.Poisson.more.precise.1.next.error}
        \oo(\lambda^{-2}) = \frac{1}{\sqrt{\lambda}} \phi(\delta_k) \sum_{j=4}^{s-2} \frac{\text{Polynomial of order $j$ in $\delta_k$}}{\lambda^{j/2}} + \oo(\lambda^{-(s-1)/2}).
    \end{equation}
    The lingering error $\oo(\lambda^{-(s-1)/2})$ means that if we try to divide both sides of \eqref{eq:LLT.Poisson.more.precise.1.next} by $\frac{1}{\sqrt{\lambda}} \phi(\delta_k)$ in order to get an approximation for the ratio $P_{\lambda}(k) / \frac{1}{\sqrt{\lambda}} \phi(\delta_k)$, then the error on the right-hand side of \eqref{eq:LLT.Poisson.more.precise.1.next.error} will become $\oo(\lambda^{-(s-2)/2} e^{\delta_k^2 / 2})$, which gives a very poor control when $\delta_k$ is large.
    In fact, if we were to use this estimate for our application in Section~\ref{sec:deficiency.distance} (look at Equations~\eqref{eq:I.plus.II.plus.III} and \eqref{eq:estimate.I.begin} in the proof of Theorem~\ref{thm:prelim.Carter}), we would need to control the error $\oo(\lambda^{-(s-2)/2} \, \EE[e^{\delta_K^2 / 2}])$ as $\lambda\to \infty$, where $K\sim \mathrm{Poisson}\hspace{0.2mm}(\lambda)$. But clearly, this cannot work since $\delta_K$ is roughly a standard normal, say $Z$, and we have $\EE[e^{Z^2 / 2}] = \infty$.

    \vspace{3mm}
    Alternatively, we have the following lemma to approximate the ratio $P_{\lambda}(k) / \frac{1}{\sqrt{\lambda}} \phi(\delta_k)$.
    It only holds in the {\it bulk} of the Poisson distribution, but this is enough for our purpose and its proof is much more accessible.
    Notice that the error in \eqref{lem:LLT.Poisson.expansion.log} is a finite polynomial in $\delta_k$, which makes it now possible to control $\OO(\lambda^{-3/2} \, \EE[1 + |\delta_K|^5])$ in the proof of Theorem~\ref{thm:prelim.Carter}.

    \begin{lemma}[Local limit theorem]\label{lem:LLT.Poisson}
        For all $\lambda > 0$, let
        \begin{equation}\label{lem:LLT.Poisson.def}
            \phi(y) \leqdef \frac{1}{\sqrt{2\pi}} e^{-y^2 / 2} \quad \text{and recall} \quad \delta_k \leqdef \frac{k - \lambda}{\sqrt{\lambda}}.
        \end{equation}
        Let $\eta\in (0,1)$ be given, possibly depending on $\lambda$.
        Then, as $\lambda\to \infty$, and uniformly for $k\in \N_0$ such that $\eta - 1 \leq \frac{\delta_k}{\sqrt{\lambda}}$, we have
        \begin{equation}\label{lem:LLT.Poisson.expansion.log}
            \begin{aligned}
                \log\bigg(\frac{P_{\lambda}(k)}{\frac{1}{\sqrt{\lambda}} \phi(\delta_k)}\bigg)
                &= \frac{1}{\sqrt{\lambda}} \Big(\frac{1}{6} \delta_k^3 - \frac{1}{2} \delta_k\Big) + \frac{1}{\lambda} \Big(-\frac{1}{12} \delta_k^4 + \frac{1}{4} \delta_k^2 - \frac{1}{12}\Big) + \OO\bigg(\frac{1 + |\delta_k|^5}{\lambda^{3/2} \eta^4}\bigg).
            \end{aligned}
        \end{equation}
        Furthermore, let $B > 0$ be given.
        Then, as $\lambda\to \infty$, and uniformly for $k\in \N_0$ such that $|\delta_k| \leq \lambda^{1/6} B$, we have
        \begin{equation}\label{lem:LLT.Poisson.expansion}
            \begin{aligned}
                \frac{P_{\lambda}(k)}{\frac{1}{\sqrt{\lambda}} \phi(\delta_k)} = 1
                &+ \frac{1}{\sqrt{\lambda}} \Big(\frac{1}{6} \delta_k^3 - \frac{1}{2} \delta_k\Big) + \frac{1}{\lambda} \Big(\frac{1}{72} \delta_k^6 - \frac{1}{6} \delta_k^4 + \frac{3}{8} \delta_k^2 - \frac{1}{12}\Big) + \OO\bigg(e^{B^3} \, \frac{1 + |\delta_k|^9}{\lambda^{3/2}}\bigg).
            \end{aligned}
        \end{equation}
    \end{lemma}

    \begin{remark}
        The advantage of the expansion in \eqref{eq:LLT.Poisson.more.precise.1} is that it is more precise to estimate $P_{\lambda}(k)$ alone if the control of the ratio $P_{\lambda}(k) / \frac{1}{\sqrt{\lambda}} \phi(\delta_k)$ is not needed. It also follows from a more general Fourier method for lattice random variables.
        The advantage of Lemma~\ref{lem:LLT.Poisson} however is that the methods are far more elementary (we only appeal to Taylor expansions and Stirling's formula) and the proof is otherwise completely self-contained, so the reader can easily verify its validity.
        In contrast, the approximation \eqref{eq:LLT.Poisson.more.precise.1} ultimately comes from difficult estimates of Fourier analysis found in \cite{MR14626}, which itself relies on results from \cite{doi:10.1080/03461238.1928.10416862,Cramer_1937_book}.
        Note that the approximations in Lemma~\ref{lem:LLT.Poisson} are not strictly comparable to Equations \eqref{eq:LLT.Poisson.more.precise.1} and \eqref{eq:LLT.Poisson.more.precise.1.next}.
        This is because Lemma~\ref{lem:LLT.Poisson} approximates the ratio $P_{\lambda}(k) / \frac{1}{\sqrt{\lambda}} \phi(\delta_k)$ and the $\oo(\lambda^{-2})$ error in Equations \eqref{eq:LLT.Poisson.more.precise.1} and \eqref{eq:LLT.Poisson.more.precise.1.next} is not relative to $\phi(\delta_k)$, which gives a very poor control when trying to approximate the ratio.
        As explained below \eqref{eq:LLT.Poisson.more.precise.1.next.error}, it doesn't seem possible to derive an approximation of $P_{\lambda}(k) / \frac{1}{\sqrt{\lambda}} \phi(\delta_k)$, with a good control on the error, from the results presented in \cite[Chapter IV]{MR14626} and \cite{MR207011}.
        The approximation of the ratio turns out to be crucial for our application in Section~\ref{sec:deficiency.distance}; this becomes clear if one looks at Equations~\eqref{eq:I.plus.II.plus.III} and \eqref{eq:estimate.I.begin} in the proof of Theorem~\ref{thm:prelim.Carter}.
    \end{remark}

    \begin{proof}
        By expanding the left-hand side of \eqref{lem:LLT.Poisson.expansion.log}, we have
        \begin{equation}
            \log\bigg(\frac{P_{\lambda}(k)}{\frac{1}{\sqrt{\lambda}} \phi(\delta_k)}\bigg) = \frac{1}{2} \log (2\pi) + (k + \tfrac{1}{2}) \log \lambda - \log k! - \lambda + \frac{1}{2} \delta_k^2.
        \end{equation}
        Stirling's formula yields
        \begin{equation}
            \log k! = \frac{1}{2} \log (2\pi) + (k + \tfrac{1}{2}) \log k - k + \frac{1}{12 k} + \OO(k^{-3}),
        \end{equation}
        (see e.g.\ \cite[p.257]{MR0167642}).
        Hence, we get
        \begin{align}\label{eq:big.equation}
            \log\bigg(\frac{P_{\lambda}(k)}{\frac{1}{\sqrt{\lambda}} \phi(\delta_k)}\bigg)
            &= - k \log \bigg(\frac{k}{\lambda}\bigg) - \frac{1}{2} \log \bigg(\frac{k}{\lambda}\bigg) \notag \\[-2mm]
            &\quad+ (k - \lambda) + \frac{1}{2} \delta_k^2 - \frac{1}{12 \lambda} \bigg(\frac{k}{\lambda}\bigg)^{-1} + \OO\bigg(\frac{1}{\lambda^3} \bigg(\frac{k}{\lambda}\bigg)^{\hspace{-1mm}-3\hspace{0.2mm}}\bigg) \notag \\[1mm]
            &= - \lambda \bigg(1 + \frac{\delta_k}{\sqrt{\lambda}}\bigg) \log \bigg(1 + \frac{\delta_k}{\sqrt{\lambda}}\bigg) \notag \\
            &\quad+ (k - \lambda) + \frac{1}{2} \delta_k^2 - \frac{1}{2} \log \bigg(1 + \frac{\delta_k}{\sqrt{\lambda}}\bigg) \notag \\
            &\quad- \frac{1}{12 \lambda} \bigg(1 + \frac{\delta_k}{\sqrt{\lambda}}\bigg)^{-1} + \OO\bigg(\frac{1}{\lambda^3} \bigg(1 + \frac{\delta_k}{\sqrt{\lambda}}\bigg)^{\hspace{-1mm}-3\hspace{0.2mm}}\bigg).
        \end{align}
        Now, note that for $y \geq \eta - 1$, Lagrange's error bound for Taylor expansions yields
        \begin{equation}
            \begin{aligned}
                (1 + y) \log (1 + y) &= y + \frac{y^2}{2} - \frac{y^3}{6} + \frac{y^4}{12} + \OO\bigg(\frac{y^5}{\eta^4}\bigg), \\
                \log (1 + y) &= y - \frac{y^2}{2} + \OO\bigg(\frac{y^3}{\eta^3}\bigg), \\
                (1 + y)^{-1} &= 1 + \OO\bigg(\frac{y}{\eta^2}\bigg).
            \end{aligned}
        \end{equation}
        By applying these estimates in \eqref{eq:big.equation}, we obtain
        \begin{align}\label{eq:big.equation.2}
            \log\bigg(\frac{P_{\lambda}(k)}{\frac{1}{\sqrt{\lambda}} \phi(\delta_k)}\bigg)
            &= - \lambda \left\{\frac{\delta_k}{\sqrt{\lambda}} + \frac{1}{2} \bigg(\frac{\delta_k}{\sqrt{\lambda}}\bigg)^2 - \frac{1}{6} \bigg(\frac{\delta_k}{\sqrt{\lambda}}\bigg)^3 + \frac{1}{12} \bigg(\frac{\delta_k}{\sqrt{\lambda}}\bigg)^4 + \OO\bigg(\frac{(\delta_k / \sqrt{\lambda})^5}{\eta^4}\bigg)\right\} \notag \\
            &\quad+ (k - \lambda) + \frac{1}{2} \delta_k^2 - \frac{1}{2} \left\{\frac{\delta_k}{\sqrt{\lambda}} - \frac{1}{2} \bigg(\frac{\delta_k}{\sqrt{\lambda}}\bigg)^2 + \OO\bigg(\frac{(\delta_k / \sqrt{\lambda})^3}{\eta^3}\bigg)\right\} \notag \\
            &\quad- \frac{1}{12 \lambda} \left\{1 + \OO\bigg(\frac{\delta_k / \sqrt{\lambda}}{\eta^2}\bigg)\right\} + \OO\bigg(\frac{1}{\lambda^3 \eta^3}\bigg).
        \end{align}
        The terms $k - \lambda$ and $\tfrac{1}{2} \delta_k^2$ cancel out with the first two terms inside the braces on the first line.
        Therefore,
        \begin{align}\label{eq:big.equation.3}
            \log\bigg(\frac{P_{\lambda}(k)}{\frac{1}{\sqrt{\lambda}} \phi(\delta_k)}\bigg)
            &= \left\{\frac{1}{6} \frac{\delta_k^3}{\sqrt{\lambda}} - \frac{1}{12} \frac{\delta_k^4}{\lambda} + \OO\bigg(\frac{\delta_k^5}{\lambda^{3/2} \eta^4}\bigg)\right\} \notag \\[-0.5mm]
            &\quad+ \left\{-\frac{1}{2} \frac{\delta_k}{\sqrt{\lambda}} + \frac{1}{4} \frac{\delta_k^2}{\lambda} + \OO\bigg(\frac{\delta_k^3}{\lambda^{3/2} \eta^3}\bigg)\right\} - \frac{1}{12 \lambda} + \OO\bigg(\frac{1 + |\delta_k|}{\lambda^{3/2} \eta^3}\bigg) \notag \\[2mm]
            &= \frac{\frac{1}{6} \delta_k^3 - \frac{1}{2} \delta_k}{\sqrt{\lambda}} + \frac{-\frac{1}{12} \delta_k^4 + \frac{1}{4} \delta_k^2 - \frac{1}{12}}{\lambda} + \OO\bigg(\frac{1 + |\delta_k|^5}{\lambda^{3/2} \eta^4}\bigg),
        \end{align}
        which proves \eqref{lem:LLT.Poisson.expansion.log}.
        To obtain \eqref{lem:LLT.Poisson.expansion} and conclude the proof, we take the exponential on both sides of the last equation with $\eta = 1 - \lambda^{-1/3} B$, and we expand the right-hand side with
        \begin{equation}\label{eq:Taylor.exponential}
            e^y = 1 + y + \frac{y^2}{2} + \OO(e^{\widetilde{B}} y^3), \quad \text{for } -\infty < y \leq \widetilde{B}.
        \end{equation}
        Note that, for $\lambda$ large enough and uniformly for $|\delta_k| \leq \lambda^{1/6} B$, the last line in \eqref{eq:big.equation.3} satisfies $|\eqref{eq:big.equation.3}| \leq \frac{1 + |\delta_k|^3}{\sqrt{\lambda}} \leq 1 + B^3$.
        When this bound is taken as $y$ in \eqref{eq:Taylor.exponential}, it explains the error in \eqref{lem:LLT.Poisson.expansion}.
        This ends the proof.
    \end{proof}

\section{Applications}\label{sec:applications}

    In this section, we present two applications of Lemma~\ref{lem:LLT.Poisson} related to the Le Cam distance between Poisson and Gaussian experiments (Section~\ref{sec:deficiency.distance}) and asymptotic properties of Bernstein estimators with Poisson weights (Section~\ref{sec:Bernstein.estimators}).

    \subsection{The Le Cam distance between Poisson and Gaussian experiments}\label{sec:deficiency.distance}

        In \cite{MR1922539}, the author finds an upper bound on the Le Cam distance (called $\Delta$-distance in \cite{MR1784901})
        between multinomial and multivariate Gaussian experiments.
        Carter achieves his goal by looking at the total variation between the measure of a multinomial vector for which the components were jittered by uniforms and a multivariate Gaussian measure with the same covariance profile.
        Carter's bound was later improved in \cite{arXiv:2001.08512} by developing a precise local limit theorem for the multinomial distribution (analogous to Lemma~\ref{lem:LLT.Poisson}) and by applying it to remove the inductive part of Carter's argument.
        We use Carter's strategy below by jittering a Poisson random variable to obtain an upper bound on the Le Cam distance between (unidimensional) Poisson and Gaussian experiments.
        For an excellent and concise review on Le Cam's theory for the comparison of statistical models, we refer the reader to \cite{MR3850766}.

        \vspace{3mm}
        The following theorem is analogous to Lemma~2 in \cite{MR1922539}.
        It bounds the total variation between a $\mathrm{Poisson}\hspace{0.2mm}(\lambda)$ random variable and a $\mathrm{Normal}\hspace{0.3mm}(\lambda, \lambda)$ random variable.
        The Le Cam bound appears in Theorem~\ref{thm:bound.deficiency.distance} right after.
        A reviewer pointed out to us that a proof of this result was already sketched, using a different approach and relying on multiple exercises with no solutions, in Section 1.4 of an unpublished set of lecture notes by \cite{Pollard_2010_Beijing_Chap_1}. Below, we give a complete proof.

        \begin{theorem}\label{thm:prelim.Carter}
            Let $K\sim \mathrm{Poisson}\hspace{0.2mm}(\lambda)$ and $U\sim \mathrm{Uniform}\hspace{0.2mm}(-\tfrac{1}{2},\tfrac{1}{2})$, where $K$ and $U$ are assumed independent.
            Define $Y \leqdef K + U$ and let $\widetilde{\PP}_{\lambda}$ be the law of $Y$.
            In particular, if $\PP_{\lambda}$ is the law of $K$, note that
            \begin{equation}
                \widetilde{\PP}_{\lambda}(B) = \int_{\N_0} \int_{(-\frac{1}{2},\frac{1}{2})} \ind_{B}(k + u) \rd u \, \PP_{\lambda}(\rd k), \quad B\in \mathscr{B}(\R).
            \end{equation}
            Let $\QQ_{\lambda}$ be the law of a $\mathrm{Normal}\hspace{0.3mm}(\lambda, \lambda)$ random variable.
            Then, for $\lambda > 0$ large enough, we have
            \begin{equation}
                \|\widetilde{\PP}_{\lambda} - \QQ_{\lambda}\| \leq \frac{C}{\sqrt{\lambda}},
            \end{equation}
            where $\| \cdot \|$ denotes the total variation norm and $C > 0$ is a universal constant.
        \end{theorem}

        \begin{proof}
            By the comparison of the total variation norm with the Hellinger distance on page 726 of \cite{MR1922539}, we already know that, for any Borel set $A\in \mathscr{B}(\R)$,
            \begin{equation}\label{eq:first.bound.total.variation}
                \|\widetilde{\PP}_{\lambda} - \QQ_{\lambda}\| \leq \sqrt{2 \, \PP(Y\in A^c) + \EE\bigg[\log\bigg(\frac{\rd \widetilde{\PP}_{\lambda}}{\rd \QQ_{\lambda}}(Y)\bigg) \, \ind_{\{Y\in A\}}\bigg]}.
            \end{equation}
            The idea is to choose a set $A^c$ that excludes the bulk of the Poisson distribution so that the probability $\PP(Y\in A^c)$ is small by a large deviation bound.
            Precisely, we can take the set
            \begin{equation}
                A^c \leqdef \big\{y\in (-\tfrac{1}{2},\infty) : |y - \lambda| >  \lambda^{2/3}\big\}.
            \end{equation}
            By applying a standard tail bound for the $\mathrm{Poisson}\hspace{0.2mm}(\lambda)$ distribution (see, e.g., \cite[p.26]{Pollard_2015_Mini_Empirical}), we have, for all $\lambda \geq 1$,
            \begin{equation}\label{eq:large.deviation.bound}
                \PP(Y\in A^c) \leq \PP\big(|K - \lambda| > \lambda^{2/3} - 1/2\big) \leq 100 \, \exp\Big(-\frac{\lambda^{1/3}}{100}\Big).
            \end{equation}
            For the expectation in \eqref{eq:first.bound.total.variation}, if $y \mapsto \widetilde{P}_{\lambda}(y)$ denotes the density function associated with $\widetilde{\PP}_{\lambda}$ (i.e.\ it is equal to $P_{\lambda}(k)$ whenever $k\in \N_0$ is closest to $y$), then
            \begin{align}\label{eq:I.plus.II.plus.III}
                \EE\bigg[\log\bigg(\frac{\rd \widetilde{\PP}_{\lambda}}{\rd \QQ_{\lambda}}(Y)\bigg) \, \ind_{\{Y\in A\}}\bigg]
                &=\EE\bigg[\log\bigg(\frac{\widetilde{P}_{\lambda}(Y)}{\frac{1}{\sqrt{\lambda}} \phi(\delta_{Y})}\bigg) \, \ind_{\{Y\in A\}}\bigg] \notag \\[1mm]
                &= \EE\bigg[\log\bigg(\frac{P_{\lambda}(K)}{\frac{1}{\sqrt{\lambda}} \phi(\delta_{K})}\bigg) \, \ind_{\{K\in A\}}\bigg] \notag \\
                &\quad+ \EE\bigg[\log\bigg(\frac{\frac{1}{\sqrt{\lambda}} \phi(\delta_{K})}{\frac{1}{\sqrt{\lambda}} \phi(\delta_{Y})}\bigg) \, \ind_{\{K\in A\}}\bigg] \notag \\[1mm]
                &\quad+ \EE\bigg[\log\bigg(\frac{P_{\lambda}(K)}{\frac{1}{\sqrt{\lambda}} \phi(\delta_{Y})}\bigg) \, (\ind_{\{Y\in A\}} - \ind_{\{K\in A\}})\bigg] \notag \\[1mm]
                &\reqdef (\mathrm{I}) + (\mathrm{II}) + (\mathrm{III}).
            \end{align}
            By Lemma~\ref{lem:LLT.Poisson} (note that $\big|\frac{\delta_k}{\sqrt{\lambda}}\big| \leq \frac{1}{2}$ for all $k\in A\cap \N_0$ if $\lambda$ is assumed large enough),
            \begin{equation}\label{eq:estimate.I.begin}
                \begin{aligned}
                    (\mathrm{I})
                    &= \frac{1}{\sqrt{\lambda}}  \cdot \EE\left[\bigg(\frac{1}{6} \cdot \frac{(K - \lambda)^3}{\lambda^{3/2}} - \frac{1}{2} \cdot \frac{K - \lambda}{\lambda^{1/2}}\bigg) \, \ind_{\{K\in A\}}\right] \\[1mm]
                    &\quad+ \OO\left(\frac{1}{\lambda} \Big(\frac{\EE[|K - \lambda|^4]}{\lambda^2} + \frac{\EE[|K - \lambda|^2]}{\lambda} + 1\Big)\right) + \OO(\lambda^{-3/2}).
                \end{aligned}
            \end{equation}
            By Lemma~\ref{lem:central.moments.Poisson}, the first $\OO(\cdot)$ term above is $\OO(\lambda^{-1})$.
            By Corollary~\ref{cor:central.moments.Poisson.on.events}, we can also control the $\asymp \lambda^{-1/2}$ term on the right-hand side of \eqref{eq:estimate.I.begin}.
            We obtain
            \begin{align}\label{eq:estimate.I}
                (\mathrm{I})
                &= \OO(\lambda^{-1}) + \OO\big(\lambda^{-1/2} (\PP(K\in A^c))^{1/2}\big) \notag \\[2mm]
                &= \OO\big(\lambda^{-1}\big).
            \end{align}
            For the term $(\mathrm{II})$ in \eqref{eq:I.plus.II.plus.III},
            \begin{align}\label{eq:estimate.II.prelim}
                \log\bigg(\frac{\frac{1}{\sqrt{\lambda}} \phi(\delta_{K})}{\frac{1}{\sqrt{\lambda}} \phi(\delta_{Y})}\bigg)
                &= \frac{(Y - \lambda)^2}{2 \lambda} - \frac{(K - \lambda)^2}{2 \lambda} \notag \\[-1mm]
                &= \frac{(Y - K)^2}{2 \lambda} + \frac{(Y - K) (K - \lambda)}{\lambda}.
            \end{align}
            With our assumption that $K$ and $Y - K = U\sim \mathrm{Uniform}\hspace{0.2mm}(-\tfrac{1}{2},\tfrac{1}{2})$ are independent, we get
            \begin{align}\label{eq:estimate.II}
                (\mathrm{II})
                &= \frac{1/12}{2 \lambda} - \frac{\EE[(Y - K)^2 \, \ind_{\{K\in A^c\}}]}{2 \lambda} - \frac{\EE[(Y - K) (K - \lambda) \, \ind_{\{K\in A^c\}}]}{\lambda} \notag \\[1mm]
                &= \frac{1}{24 \lambda} + \OO\bigg(\frac{\PP(K\in A^c)}{\lambda}\bigg) + \OO\bigg(\frac{\sqrt{\EE[(K - \lambda)^2]} \, \sqrt{\PP(K\in A^c)}}{\lambda}\bigg) \notag \\[2mm]
                &= \OO(\lambda^{-1}).
            \end{align}
            For the term $(\mathrm{III})$ in \eqref{eq:I.plus.II.plus.III}, we have the following crude bound from Lemma~\ref{lem:LLT.Poisson} and \eqref{eq:estimate.II.prelim}, on the symmetric difference event $\{Y\hspace{-0.5mm}\in\hspace{-0.5mm} A\} \hspace{0.3mm}\triangle\hspace{0.3mm} \{K\hspace{-0.5mm}\in\hspace{-0.5mm} A\}$,\hspace{0.5mm}%
            \footnote{On $\{Y\hspace{-0.5mm}\in\hspace{-0.5mm} A\} \hspace{0.3mm}\triangle\hspace{0.3mm} \{K\hspace{-0.5mm}\in\hspace{-0.5mm} A\}$, the condition $\big|\frac{\delta_K}{\sqrt{\lambda}}\big| \leq \frac{1}{2}$ is easily satisfied if $\lambda$ is assumed large enough, simply because $|Y - K| \leq \frac{1}{2}$.}
            \begin{align}
                \log\bigg(\frac{P_{\lambda}(K)}{\frac{1}{\sqrt{\lambda}} \phi(\delta_{Y})}\bigg)
                &= \log\bigg(\frac{P_{\lambda}(K)}{\frac{1}{\sqrt{\lambda}} \phi(\delta_{K})}\bigg) + \log\bigg(\frac{\frac{1}{\sqrt{\lambda}} \phi(\delta_{K})}{\frac{1}{\sqrt{\lambda}} \phi(\delta_{Y})}\bigg) \notag \\
                &= \OO\bigg(\frac{|K - \lambda|^3}{\lambda^2} + \frac{|K - \lambda|}{\lambda} + \frac{1}{\lambda}\bigg),
            \end{align}
            which yields, by Cauchy-Schwarz and Lemma~\ref{lem:central.moments.Poisson},
            \begin{align}\label{eq:estimate.III}
                (\mathrm{III})
                &= \OO\left(\sqrt{\frac{\EE[|K - \lambda|^6]}{\lambda^4} + \frac{\EE[|K - \lambda|^2] + 1}{\lambda^2}} \, \sqrt{\PP\big(\{Y\in A\} \triangle \{K\in A\}\big)}\right) \notag \\[1mm]
                &= \OO\left(\lambda^{-1/2} \, \sqrt{\PP\big(\{Y\in A\} \triangle \{K\in A\}\big)}\right).
            \end{align}
            Putting \eqref{eq:estimate.I}, \eqref{eq:estimate.II} and \eqref{eq:estimate.III} in \eqref{eq:I.plus.II.plus.III}, together with the exponential bound
            \begin{align}
                \PP\big(\{Y\in A\} \triangle \{K\in A\}\big)
                &\leq \PP(K\in A^c) + \PP(Y\in A^c) \notag \\[-1mm]
                &\leq 2 \cdot 100 \, \exp\Big(-\frac{\lambda^{1/3}}{100}\Big),
            \end{align}
            yields, as $\lambda\to \infty$,
            \begin{equation}\label{eq:I.plus.II.plus.III.end}
                \EE\bigg[\log\bigg(\frac{\rd \widetilde{\PP}_{\lambda}}{\rd \QQ_{\lambda}}(Y)\bigg) \, \ind_{\{Y\in A\}}\bigg] = (\mathrm{I}) + (\mathrm{II}) + (\mathrm{III}) = \OO(\lambda^{-1}).
            \end{equation}
            Now, putting \eqref{eq:large.deviation.bound} and \eqref{eq:I.plus.II.plus.III.end} together in \eqref{eq:first.bound.total.variation} gives the conclusion.
        \end{proof}

        By inverting the Markov kernel that jitters the Poisson random variable, we obtain the aforementioned Le Cam distance upper bound between Poisson and Gaussian experiments.

        \begin{theorem}[Bound on the Le Cam distance]\label{thm:bound.deficiency.distance}
            Let $\lambda_0 > 0$ be given, and define the experiments
            \vspace{-2mm}
            \begin{alignat*}{6}
                    &\mathscr{P}
                    &&\leqdef &&~\{\PP_{\lambda}\}_{\lambda \geq \lambda_0}, \quad &&\PP_{\lambda} ~\text{is the measure induced by } \mathrm{Poisson}\hspace{0.2mm}(\lambda), \\
                    &\mathscr{Q}\hspace{-0.5mm}
                    &&\leqdef &&~\{\QQ_{\lambda}\}_{\lambda \geq \lambda_0}, \quad &&\QQ_{\lambda} ~\text{is the measure induced by } \mathrm{Normal}\hspace{0.3mm}(\lambda, \lambda).
            \end{alignat*}
            Then, we have the following bound on the Le Cam distance $\Delta(\mathscr{P},\mathscr{Q})$ between $\mathscr{P}$ and $\mathscr{Q}$,
            \begin{equation}\label{eq:thm:bound.deficiency.distance.bound}
                \Delta(\mathscr{P},\mathscr{Q}) \leqdef \max\{\delta(\mathscr{P},\mathscr{Q}),\delta(\mathscr{Q},\mathscr{P})\} \leq \frac{C}{\sqrt{\lambda_0}},
            \end{equation}
            where $C > 0$ is a universal constant,
            \begin{equation}\label{eq:def:deficiency.one.sided}
                \begin{aligned}
                    \delta(\mathscr{P},\mathscr{Q})
                    &\leqdef \inf_{T_1} \sup_{\lambda \geq \lambda_0} \bigg\|\int_{\N_0} T_1(k, \cdot \, ) \, \PP_{\lambda}(\rd k) - \QQ_{\lambda}\bigg\|, \\
                    \delta(\mathscr{Q},\mathscr{P})
                    &\leqdef \inf_{T_2} \sup_{\lambda \geq \lambda_0} \bigg\|\PP_{\lambda} - \int_{\R} T_2(y, \cdot \, ) \, \QQ_{\lambda}(\rd y)\bigg\|, \\
                \end{aligned}
            \end{equation}
            and the infima are taken, respectively, over all Markov kernels $T_1 : \N_0 \times \mathscr{B}(\R) \to [0,1]$ and $T_2 : \R \times \mathscr{B}(\N_0) \to [0,1]$.
        \end{theorem}

        \begin{proof}
            By Theorem~\ref{thm:prelim.Carter}, we get the desired bound on $\delta(\mathscr{P},\mathscr{Q})$ by choosing the Markov kernel $T_1^{\star}$ that adds $U$ to $K$, namely
            \begin{equation}
                \begin{aligned}
                    T_1^{\star}(k,B) \leqdef \int_{(-\frac{1}{2},\frac{1}{2})} \ind_{B}(k + u) \rd u, \quad k\in \N_0, ~B\in \mathscr{B}(\R).
                \end{aligned}
            \end{equation}
            To get the bound on $\delta(\mathscr{Q},\mathscr{P})$, it suffices to consider a Markov kernel $T_2^{\star}$ that inverts the effect of $T_1^{\star}$, i.e.\ rounding off $Z\sim \mathrm{Normal}\hspace{0.3mm}(\lambda, \lambda)$ to the nearest integer.
            Then, as explained in Section~5 of \cite{MR1922539}, we get
            \begin{align}
                \delta(\mathscr{Q},\mathscr{P})
                &\leq \sup_{\lambda\geq \lambda_0} \bigg\|\PP_{\lambda} - \int_{\R} T_2^{\star}(z, \cdot \, ) \, \QQ_{\lambda}(\rd z)\bigg\| \notag \\
                &= \sup_{\lambda\geq \lambda_0} \bigg\|\int_{\R} T_2^{\star}(z, \cdot \, ) \int_{\N_0} T_1^{\star}(k, \rd z) \, \PP_{\lambda}(\rd k) - \int_{\R} T_2^{\star}(z, \cdot \, ) \, \QQ_{\lambda}(\rd z)\bigg\| \notag \\
                &\leq \sup_{\lambda\geq \lambda_0} \bigg\|\int_{\N_0} T_1^{\star}(k, \cdot \, ) \, \PP_{\lambda}(\rd k) - \QQ_{\lambda}\bigg\|,
            \end{align}
            and we get the same bound by Theorem~\ref{thm:prelim.Carter}.
        \end{proof}

        If we consider the following Gaussian experiment with constant variance

        \vspace{-3mm}
        \begin{alignat*}{6}
            &\mathscr{Q}^{\star}\hspace{-0.5mm}
            &&\leqdef &&~\{\QQ_{\lambda}^{\star}\}_{\lambda \geq \lambda_0}, \quad &&\QQ_{\lambda}^{\star} ~\text{is the measure induced by } \mathrm{Normal}\hspace{0.3mm}(\sqrt{\lambda}, 1/4),
        \end{alignat*}

        \vspace{2mm}
        \noindent
        then the argument in \cite[Section~7]{MR1922539} can be adapted to show that
        \begin{equation}\label{eq:LeCam.distance.indep.normals}
            \Delta(\mathscr{Q},\mathscr{Q}^{\star}) \leq \frac{C}{\sqrt{\lambda_0}},
        \end{equation}
        using a variance stabilizing transformation,
        with proper adjustments to the deficiencies in \eqref{eq:def:deficiency.one.sided}.
        As a direct consequence, we obtain the following corollary.

        \begin{corollary}\label{cor:main.theorem.consequence}
            With the same notation as in Theorem~\ref{thm:bound.deficiency.distance}, we have
            \begin{equation}
                \Delta(\mathscr{P},\mathscr{Q}^{\star}) \leq \frac{C}{\sqrt{\lambda_0}},
            \end{equation}
            where $C > 0$ is a universal constant.
        \end{corollary}

        \begin{proof}
            This follows from Theorem~\ref{thm:bound.deficiency.distance}, Equation \eqref{eq:LeCam.distance.indep.normals} and the triangle inequality for the pseudometric $\Delta(\cdot\hspace{0.3mm}, \cdot)$.
        \end{proof}

    \subsection{Asymptotic properties of Bernstein estimators with Poisson weights}\label{sec:Bernstein.estimators}

        Bernstein estimators generally refer to a class of estimators for density functions and cumulative distribution functions (c.d.f.s) where a weight is added to the empirical density or empirical c.d.f.\ to produce a smooth and a variable smoothing across the support of the target.
        These estimators are known to behave better than traditional kernel estimators (see, e.g., \cite{MR79873} and \cite{MR143282}) near the boundary of the support.
        The weight is always a discrete probability mass function seen as a function of one of its parameters.
        For instance, the weights are usually binomial (controlled by $p$) and Poisson (controlled by $\lambda$) when the support of the target is $[0,1]$ and $[0,\infty)$, respectively.
        Asymptotic properties of Bernstein estimators on compacts supports have been studied by many authors, see e.g.\ \cite{MR0397977,MR0638651,MR0726014,MR1293514,MR1910059,MR2270097,MR2351744,MR2662607,MR2960952,MR2925964,MR3174309,MR3474765,MR3983257,arXiv:2002.07758,arXiv:2006.11756}, just to name a few. For the interested reader, there is an extensive review in Section 2 of \cite{arXiv:2002.07758}.

        \vspace{3mm}
        In the literature, various asymptotic properties of Bernstein estimators with Poisson weights were also studied, namely: \cite{MR0574548,MR0638651} studied the bias, variance and mean squared error for the density estimator, and \cite{arXiv:2005.09994} studied the bias, variance, mean squared error, mean integrated squared error, asymptotic normality, uniform strong consistency and relative deficiency with respect to the empirical c.d.f.\ for the c.d.f.\ estimator.
        The estimators are defined as follows.
        Assume that the observations $X_1, X_2, \dots, X_n$ are independent, $F$ distributed (with density $f$) and supported on $[0,\infty)$.
        Then, for $n, m\in \N$, let
        \begin{equation}\label{eq:cdf.Bernstein.estimator}
            \hat{F}_{m,n}^S(x) \leqdef \sum_{k=0}^{\infty} \left\{\frac{1}{n} \sum_{i=1}^n \ind_{(-\infty, \frac{k}{m}]}(X_i)\right\} V_{k,m}(x), \quad x\geq 0,
        \end{equation}
        be the Bernstein c.d.f.\ estimator with the Poisson weights
        \begin{equation}
            V_{k,m}(x) \leqdef P_{m x}(k) = \frac{(m x)^k}{k!} e^{-m x}, \quad k \in \N_0,
        \end{equation}
        ($\hat{F}_{m,n}^S$ was introduced in \cite{arXiv:2005.09994} as the {\it Szasz estimator}, since it can be seen as the Szasz-Mirakyan operator applied to the empirical c.d.f.), and let
        \begin{equation}\label{eq:histogram.estimator}
            \hat{f}_{m,n}^S(x) \leqdef \sum_{k=0}^{\infty} \left\{\frac{m}{n} \sum_{i=1}^n \ind_{(\frac{k}{m}, \frac{k + 1}{m}]}(X_i)\right\} V_{k,m}(x), \quad x\geq 0,
        \end{equation}
        be the Bernstein density estimator with Poisson weights (which was introduced in \cite{MR0574548}).
        Assuming that $F$ and $f$ are respectively two-times and one-time continuously differentiable on $(0,\infty)$, the theorem below computes the asymptotics of the variance of both estimators. The variance of $\hat{f}_{m,n}^S$ was previously calculated using known asymptotics of modified Bessel functions of the first kind in \cite{MR0574548}, so \eqref{eq:thm:asymptotics.variance.density} gives an alternative proof.
        As for the variance of $\hat{F}_{m,n}^S$, it was incorrectly stated in the Theorem~5 of \cite[arXiv v.1]{arXiv:2005.09994}, so \eqref{eq:thm:asymptotics.variance.cdf} below fixes their original statement and proof.
        The subsequent arXiv versions of \cite{arXiv:2005.09994} were corrected following Theorem~\ref{thm:asymptotics.variance} and Lemma~\ref{lem:technical.sums.R} of the present paper.

        \begin{remark}\label{rem:incorrect.Leblanc}
            The error in the asymptotic expression of the variance in Theorem~5 of \cite[arXiv v.1]{arXiv:2005.09994} ultimately stems from the authors' adaptation of a method from \cite{MR2960952} to estimate the technical sum
            \vspace{1mm}
            \begin{equation*}
                \sum_{0 \leq k < \ell < \infty} (\tfrac{k}{m} - x) V_{k,m}(x) V_{\ell,m}(x).
            \end{equation*}
            In Lemma~2 (iv) of \cite{MR2960952}, a continuity correction from \cite{MR538319} was used to claim that
            \vspace{-1mm}
            \begin{equation}\label{eq:claim.Leblanc.2012.Lemma.2.4}
                \sum_{0 \leq k < \ell \leq m} (\tfrac{k}{m} - x) P_{k,m}(x) P_{\ell,m}(x) + \sqrt{\frac{x (1 - x)}{2\pi}} = \oo_x(1),
            \end{equation}
            where $P_{k,m}(x) \leqdef \binom{m}{k} x^k (1 - x)^{n-k}$ are Binomial weights.
            Unfortunately, there is an error in Leblanc's proof.
            Indeed, \cite[pp.935-936]{MR2960952} applied the expansion
            \begin{equation}\label{eq:expansion.Phi}
                \Phi(t) = \frac{1}{2} + \frac{t}{\sqrt{2\pi}} + \oo(|t|)
            \end{equation}
            (which is valid for small $t$'s) to the following result developed in \cite{MR538319}:
            \begin{equation}
                \sum_{\ell = k + 1}^m P_{\ell,m}(x) = 1 - \Phi(\delta_{k+1} - G_x(\delta_{k + 1/2})) + \OO_x(m^{-1}),
            \end{equation}
            where $G_x(\cdot)$ is a specific function and $\delta_{k+1} = (k + 1 - m x) [m x (1 - x)]^{-1/2}$. The problem is that $\delta_{k+1} - G_x(\delta_{k + 1/2})$ can go up to order $\asymp_x \sqrt{m}$ (because $k+1$ goes up to $m$ in the sum), which is way beyond the range where the expansion \eqref{eq:expansion.Phi} is valid. As a consequence, there are hidden contributions in Equation (9) of \cite{MR2960952}, which results in the discrepancy between \eqref{eq:claim.Leblanc.2012.Lemma.2.4} and \eqref{eq:claim.Leblanc.2012.Lemma.2.4.correction}, see Figure~\ref{fig:graph.error.mathematica}.

            \vspace{2mm}
            A quick simulation with \texttt{Mathematica} shows that the correct estimate is
            \vspace{-1mm}
            \begin{equation}\label{eq:claim.Leblanc.2012.Lemma.2.4.correction}
                \sum_{0 \leq k < \ell \leq m} (\tfrac{k}{m} - x) P_{k,m}(x) P_{\ell,m}(x) + \sqrt{\frac{x (1 - x)}{4\pi}} = \oo_x(1).
            \end{equation}
            The mathematical proof of (a generalization of) Equation~\eqref{eq:claim.Leblanc.2012.Lemma.2.4.correction} was also given in Lemma A.3 of \cite{arXiv:2002.07758} by applying a local limit theorem for the multinomial distribution developed in the same article.
            In Appendix~\ref{sec:error.propagation}, a list of articles and theses whose statements (and proofs) have been affected by this error is collected and appropriate fixes are suggested.
        \end{remark}

        \begin{figure}[ht]
            \centering
            \includegraphics[width=100mm]{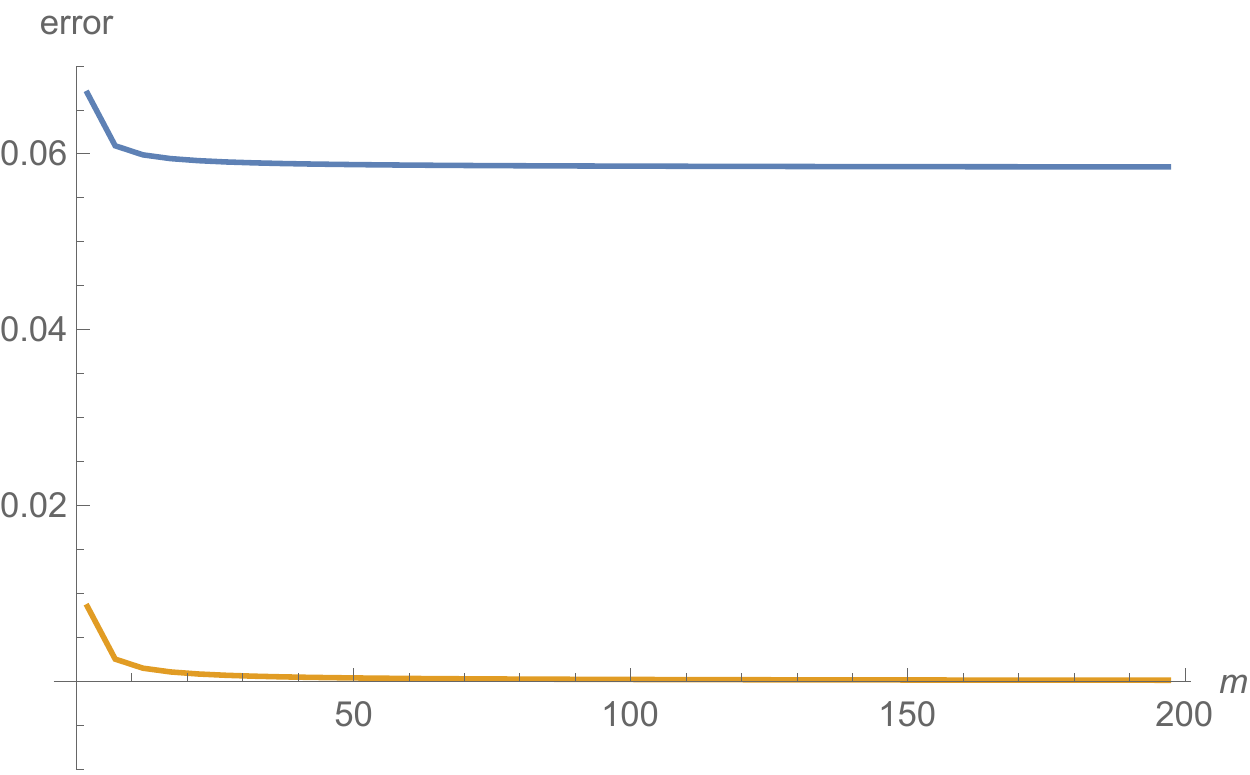}
            \caption{Graph of \eqref{eq:claim.Leblanc.2012.Lemma.2.4} (top curve) and \eqref{eq:claim.Leblanc.2012.Lemma.2.4.correction} (bottom curve) for $x = 1/2$.}
            \label{fig:graph.error.mathematica}
        \end{figure}

        Here are the asymptotics of the variance for the Bernstein c.d.f.\ and density estimators with Poisson weights (or Szasz estimators).

        \begin{theorem}\label{thm:asymptotics.variance}
            As $n\to \infty$, we have
            \begin{align}
                \VV(\hat{F}_{m,n}^S(x)) &= n^{-1} F(x) (1 - F(x)) - n^{-1} m^{-1/2} \sqrt{\frac{x}{\pi}} + \oo_x(n^{-1} m^{-1/2}), \label{eq:thm:asymptotics.variance.cdf} \\
                \VV(\hat{f}_{m,n}^S(x)) &= n^{-1} m^{1/2} \frac{f(x)}{\sqrt{4 \pi x}} + \oo_x(n^{-1} m^{1/2}). \label{eq:thm:asymptotics.variance.density}
            \end{align}
        \end{theorem}

        \newpage
        In a first step, we need to estimate some technical sums using our local limit theorem.

        \begin{lemma}\label{lem:technical.sums.R}
            Let
            \begin{equation}
                \widetilde{R}_{1,m}^S \leqdef m^{1/2} \sum_{k,\ell=0}^{\infty} (\tfrac{k \wedge \ell}{m} - x) V_{k,m}(x) V_{\ell,m}(x), \quad x\in (0,\infty).
            \end{equation}
            Then, for any given $x\in (0,\infty)$,
            \begin{equation}\label{eq:lem:technical.sums.R.claim.1}
                \sup_{m\in \N} |\widetilde{R}_{1,m}^S(x)| \leq 2 \sqrt{x},
            \end{equation}
            and we have, as $m\to \infty$,
            \begin{align}
                m^{1/2} \sum_{k=0}^{\infty} V_{k,m}^2(x) &= \frac{1}{\sqrt{4 \pi x}} + \oo_x(1), \label{eq:lem:technical.sums.R.claim.2.1} \\
                m \sum_{k=0}^{\infty} V_{k,m}^3(x) &= \frac{1}{2\sqrt{3} \pi x} + \oo_x(1), \label{eq:lem:technical.sums.R.claim.2.2} \\
                \widetilde{R}_{1,m}^S(x) &= -\sqrt{\frac{x}{\pi}} + \oo_x(1). \label{eq:lem:technical.sums.R.claim.2.3}
            \end{align}
        \end{lemma}

        \begin{proof}
            By Jensen's inequality, we have, for all $m\in \N$,
            \begin{align}
                |\widetilde{R}_{1,m}^S(x)|
                &\leq 2 m^{-1/2} \sum_{k=0}^{\infty} |k - m x| V_{k,m}(x) \notag \\
                &\leq 2 m^{-1/2} \hspace{-1mm} \sqrt{\sum_{k=0}^{\infty} |k - m x|^2 V_{k,m}(x)} \notag \\
                &= 2 m^{-1/2} \cdot \sqrt{m x} \leq 2 \sqrt{x},
            \end{align}
            which proves \eqref{eq:lem:technical.sums.R.claim.1}.
            In order to prove \eqref{eq:lem:technical.sums.R.claim.2.1}, consider the decomposition
            \begin{equation}\label{eq:tail.bound}
                m^{1/2} \sum_{k=0}^{\infty} V_{k,m}^2(x) = m^{1/2} \hspace{-5mm}\sum_{k\in \N_0 : |\delta_k| \leq (m x)^{1/6}} \hspace{-5mm} V_{k,m}^2(x) ~+~ m^{1/2} \hspace{-5mm}\sum_{k\in \N_0 : |\delta_k| > (m x)^{1/6}} \hspace{-5mm} V_{k,m}^2(x).
            \end{equation}
            The second term on the right-hand side of \eqref{eq:tail.bound} is exponentially small in $(m x)^{1/3}$ by a standard large deviation bound (see \eqref{eq:large.deviation.bound}), and the first term on the right-hand side can be approximated by a Gaussian integral because of the local limit theorem we developed in Equation~\eqref{lem:LLT.Poisson.expansion} of Lemma~\ref{lem:LLT.Poisson}.

            \vspace{2mm}
            If $\phi_{\sigma^2}$ denotes the density of the $\mathrm{Normal}\hspace{0.3mm}(0,\sigma^2)$ distribution, then
            \begin{align}
                m^{1/2} \sum_{k=0}^{\infty} V_{k,m}^2(x)
                &= \int_{\R} \frac{1}{\sqrt{x}} \phi^2(z) \rd z + \oo_x(1) \notag \\
                &= \frac{2^{-1/2}}{\sqrt{2 \pi x}} \int_{\R} \phi_{\frac{1}{2}}(z) \rd z + \oo_x(1) \notag \\
                &= \frac{1}{\sqrt{4 \pi x}} \cdot 1 + \oo_x(1),
            \end{align}
            which proves \eqref{eq:lem:technical.sums.R.claim.2.1}.
            By a similar argument,
            \begin{align}
                m^{1/2} \sum_{k=0}^{\infty} V_{k,m}^3(x)
                &= \int_{\R} \frac{1}{x} \phi^3(z) \rd z + \oo_x(1) \notag \\
                &= \frac{3^{-1/2}}{2 \pi x} \int_{\R} \phi_{\frac{1}{3}}(z) \rd z + \oo_x(1) \notag \\
                &= \frac{1}{2 \sqrt{3} \pi x} \cdot 1 + \oo_x(1),
            \end{align}
            which proves \eqref{eq:lem:technical.sums.R.claim.2.2}.
            Finally, to obtain the asymptotics of $R_{1,m}^S(x)$, consider the decomposition
            \begin{equation}\label{eq:R.1.m.decomposition}
                \begin{aligned}
                    \widetilde{R}_{1,m}^S(x)
                    &= 2 m^{1/2} \hspace{-2mm} \sum_{0 \leq k < \ell < \infty} \hspace{-1mm} (\tfrac{k}{m} - x) V_{k,m}(x) V_{\ell,m}(x) \\
                    &\quad+ m^{1/2} \sum_{k=0}^{\infty} (\tfrac{k}{m} - x) V_{k,m}^2(x).
                \end{aligned}
            \end{equation}
            The second term on the right-hand side of \eqref{eq:R.1.m.decomposition} is negligible by the Cauchy-Schwarz inequality and \eqref{eq:lem:technical.sums.R.claim.2.2}:
            \begin{align}\label{eq:R.1.m.decomposition.next}
                m^{1/2} \sum_{k=0}^{\infty} (\tfrac{k}{m} - x) V_{k,m}^2(x)
                &\leq m^{-1/2} \sqrt{\sum_{k=0}^{\infty} (k - m x)^2 V_{k,m}(x)} \sqrt{\sum_{k=0}^{\infty} V_{k,m}^3(x)} \notag \\
                &= m^{-1/2} \cdot \sqrt{m x} \cdot \OO_x(\sqrt{m^{-1}}) \notag \\[2mm]
                &= \OO_x(m^{-1/2}).
            \end{align}
            For the first term on the right-hand side of \eqref{eq:R.1.m.decomposition}, we can use the local limit theorem (Lemma~\ref{lem:LLT.Poisson}) and integration by parts.
            Together with \eqref{eq:R.1.m.decomposition} and \eqref{eq:R.1.m.decomposition.next}, we obtain
            \begin{align}
                \widetilde{R}_{1,m}^S(x)
                &= 2 \cdot x \int_{-\infty}^{\infty} \frac{z}{x} \, \phi_x(z) \int_z^{\infty} \phi_x(y) \rd y \rd z + \oo_x(1) \notag \\
                &= 2 \cdot x \, \Big[0 - \int_{-\infty}^{\infty} \phi_{x}^2(z) \rd z\Big] + \oo_x(1) \notag \\[1mm]
                &= \frac{- 2 x}{\sqrt{4\pi x}} \int_{-\infty}^{\infty} \phi_{\frac{1}{2} x}(z) \rd z + \oo_x(1) \notag \\[1mm]
                &= -\sqrt{\frac{x}{\pi}} + \oo_x(1).
            \end{align}
            This ends the proof.
        \end{proof}

        \begin{proof}[Proof of Theorem~\ref{thm:asymptotics.variance}]
            By the independence of the observations $X_i$, a Taylor expansion for the c.d.f.\ $F$, and the asymptotic expression for the bias in Theorem~4 of \cite{arXiv:2005.09994}, we have
            \vspace{-3mm}
            \begin{align}
                \VV(\hat{F}_{m,n}^S(x))
                &= \frac{1}{n} \left\{\sum_{k,\ell=0}^{\infty} F\bigg(\frac{k \wedge \ell}{m}\bigg) V_{k,m} V_{\ell,m} - \bigg(\sum_{k=0}^{\infty} F(k / m) V_{k,m}(x)\bigg)^2\right\} \notag \\
                &= \frac{1}{n} \left\{\hspace{-1mm}
                    \begin{array}{l}
                        F(x) (1 - F(x)) + \OO_x(m^{-1}) \\[1mm]
                        + f(x) \sum_{k,\ell=0}^{\infty} \big(\frac{k \wedge \ell}{m} - x\big) V_{k,m}(x) V_{\ell,m}(x) \\[1mm]
                        + \OO_x\Big(\sum_{k,\ell=0}^{\infty} \big|\frac{k}{m} - x\big| \big|\frac{\ell}{m} - x\big| V_{k,m}(x) V_{\ell,m}(x)\Big)
                    \end{array}
                    \hspace{-1mm}\right\}.
            \end{align}
            Now, by the Cauchy-Schwarz inequality, the fact that the variance of a $\mathrm{Poisson}\hspace{0.2mm}(m x)$ random variable is $m x$, and the estimate for $\widetilde{R}_{1,m}^S(x)$ in Lemma~\ref{lem:technical.sums.R},
            \begin{align}
                \VV(\hat{F}_{m,n}^S(x))
                &= \frac{1}{n} \left\{\hspace{-1mm}
                    \begin{array}{l}
                        F(x) (1 - F(x)) + \OO_x(m^{-1}) \\[1.5mm]
                        + m^{-1/2} f(x) \cdot \widetilde{R}_{1,m}^S(x) \\[0.5mm]
                        + \OO_x\Big(m^{-2} \sum_{k=0}^{\infty} |k - m x|^2 V_{k,m}(x)\Big)
                    \end{array}
                    \hspace{-1mm}\right\} \notag \\[1mm]
                &= \frac{1}{n} \left\{\hspace{-1mm}
                    \begin{array}{l}
                        F(x) (1 - F(x)) + \OO_x(m^{-1}) \\[1mm]
                        + m^{-1/2} f(x) \cdot \widetilde{R}_{1,m}^S(x)
                    \end{array}
                    \hspace{-1mm}\right\} \notag \\[0mm]
                &= n^{-1} F(x) (1 - F(x)) - n^{-1} m^{-1/2} \sqrt{\frac{x}{\pi}} + \oo_x(n^{-1} m^{-1/2}).
            \end{align}
            Similarly, by the independence of the observations $X_i$, a Taylor expansion for the density $f$, and the asymptotic expression for the bias in Equation (8) of \cite{MR0574548}, we have
            \begin{align}\label{eq:histogram.estimator.var.asymp}
                &\VV(\hat{f}_{m,n}^S(x)) \notag \\[1mm]
                &\quad= \frac{m^2}{n} \, \left\{\hspace{-1mm}
                    \begin{array}{l}
                        \sum_{k=0}^{\infty} \int_{(\frac{k}{m}, \frac{k + 1}{m}]} \hspace{-0.5mm} f(y) \rd y \, V_{k,m}^2(x) \\
                        -\Big(\sum_{k=0}^{\infty} \int_{(\frac{k}{m}, \frac{k + 1}{m}]} \hspace{-0.5mm} f(y) \rd y \, V_{k,m}(x)\Big)^2
                    \end{array}
                    \hspace{-1mm}\right\} \notag \\
                &\quad= n^{-1} m^{1/2} \left\{\hspace{-1mm}
                    \begin{array}{l}
                        m^{1/2} \sum_{k=0}^{\infty} \big(f(x) + \OO_x(m^{-1}) + |k / m - x|\big) V_{k,m}^2(x) \\[1mm]
                        + \OO_x(m^{-1/2})
                    \end{array}
                    \hspace{-1mm}\right\}.
            \end{align}
            By the Cauchy-Schwarz inequality, the fact that the variance of a $\mathrm{Poisson}\hspace{0.2mm}(m x)$ random variable is $m x$, and the estimates \eqref{eq:lem:technical.sums.R.claim.2.1} and \eqref{eq:lem:technical.sums.R.claim.2.2} in Lemma~\ref{lem:technical.sums.R}, we have
            \begin{align}
                \VV(\hat{f}_{m,n}^S(x))
                &= n^{-1} m^{1/2} \left\{\hspace{-1mm}
                    \begin{array}{l}
                        (f(x) + \OO_x(m^{-1})) \, m^{1/2} \sum_{k=0}^{\infty} V_{k,m}^2(x) \\[1mm]
                        + \OO_x(m^{-1/2}) \\[1mm]
                        + \OO\left(\hspace{-1mm}
                            \begin{array}{l}
                                \sqrt{m^{-2} \sum_{k=0}^{\infty} |k - m x|^2 \, V_{k,m}(x)} \\[1mm]
                                \cdot \, \sqrt{m \sum_{k=0}^{\infty} V_{k,m}^3(x)}
                            \end{array}
                            \hspace{-1mm}\right)
                    \end{array}
                    \hspace{-1mm}\right\} \notag \\
                &= n^{-1} m^{1/2} \frac{f(x)}{\sqrt{4 \pi x}} + \oo_x(n^{-1} m^{1/2}).
            \end{align}
            This ends the proof.
        \end{proof}

        From Theorem~\ref{thm:asymptotics.variance}, other asymptotic expressions can be (and were) derived such as the mean squared error in \cite{MR0574548} and \cite{arXiv:2005.09994}.
        We can also optimize the bandwidth parameter $m$ with respect to these expressions to implement a plug-in selection method, exactly as we would in the setting of traditional multivariate kernel estimators, see e.g.\ \cite[Section~6.5]{MR3329609} or \cite[Section~3.6]{MR3822372}.

\begin{appendices}

\section{Technical lemmas}

    Below, we prove a general formula for the central moments of the Poisson distribution, and evaluate the second, third, fourth and sixth central moments explicitly. This lemma is used to estimate some expectations in \eqref{eq:estimate.III} and the $\asymp \lambda^{-1}$ errors in \eqref{eq:estimate.I.begin} of the proof of Theorem~\ref{thm:prelim.Carter}.
    It is also a preliminary result for the proof of Corollary~\ref{cor:central.moments.Poisson.on.events} below, where the central moments are estimated on various events.

    \begin{lemma}[Central moments of the Poisson distribution]\label{lem:central.moments.Poisson}
        Let $K\sim \mathrm{Poisson}\hspace{0.2mm}(\lambda)$ for some $\lambda > 0$.
        The general formula for the central moments is given by
        \begin{equation}\label{eq:lem:central.moments.Poisson.main}
            \EE[(K - \lambda)^n] = \sum_{\ell=0}^n \binom{n}{\ell} \bigg(\sum_{j=0}^{\ell} \brkbinom{\ell}{j} \lambda^j\bigg) (-\lambda)^{n - \ell}, \quad n\in \N,
        \end{equation}
        where $\brkbinom{\ell}{j}$ denotes a Stirling number of the second kind.
        In particular,
        \begin{equation}\label{eq:lem:central.moments.Poisson.examples}
            \begin{aligned}
                &\EE[(K - \lambda)^2] = \lambda, \\
                &\EE[(K - \lambda)^3] = \lambda, \\
                &\EE[(K - \lambda)^4] = 3 \lambda^2 + \lambda, \\
                &\EE[(K - \lambda)^6] = 15 \lambda^3 + 25 \lambda^2 + \lambda.
            \end{aligned}
        \end{equation}
    \end{lemma}

    \begin{proof}
        It is well known that the non-central moments of the Poisson distribution are given by Touchard polynomials (see, e.g., Equation (3.4) in \cite{doi:10.2307/2957598}), namely
        \begin{equation}
            \EE\big[K^{\ell}\big] = \sum_{j=0}^{\ell} \brkbinom{\ell}{j} \lambda^j, \quad \ell\in \N.
        \end{equation}
        The conclusion \eqref{eq:lem:central.moments.Poisson.main} follows from an application of the binomial formula.
    \end{proof}

    We can also estimate the moments of Lemma~\ref{lem:central.moments.Poisson} on various events.
    The corollary below is used to estimate the $\asymp \lambda^{-1/2}$ errors in \eqref{eq:estimate.I.begin} of the proof of Theorem~\ref{thm:prelim.Carter}.

    \begin{corollary}\label{cor:central.moments.Poisson.on.events}
        Let $K\sim \mathrm{Poisson}\hspace{0.2mm}(\lambda)$ for some $\lambda > 0$, and let $A\in \mathscr{B}(\R)$ be a Borel set.
        Then,
        \begin{equation}\label{eq:cor:central.moments.Poisson.on.events.main}
            \begin{aligned}
                &\big|\EE[(K - \lambda) \, \ind_{\{K\in A\}}]\big| \leq \lambda^{1/2} (\PP(K\in A^c))^{1/2}, \\
                &\big|\EE[(K - \lambda)^2 \, \ind_{\{K\in A\}}] - \lambda\big| \leq 2 (1 + \lambda) (\PP(K\in A^c))^{1/2}, \\
                &\big|\EE[(K - \lambda)^3 \, \ind_{\{K\in A\}}] - \lambda\big| \leq \sqrt{41} (1 + \lambda)^{3/2} (\PP(K\in A^c))^{1/2}.
            \end{aligned}
        \end{equation}
    \end{corollary}

    \begin{proof}
        Note that \eqref{eq:lem:central.moments.Poisson.examples} implies
        \begin{equation}\label{eq:cor:central.moments.Poisson.on.events.proof.first}
            \begin{aligned}
                &\EE[(K - \lambda)^2] = \lambda, \\
                &\EE[(K - \lambda)^4] \leq 3 (1 + \lambda)^2 + (1 + \lambda)^2 = 4 (1 + \lambda)^2, \\
                &\EE[(K - \lambda)^6] \leq 15 (1 + \lambda)^3 + 25 (1 + \lambda)^3 + (1 + \lambda)^3 = 41 (1 + \lambda)^3.
            \end{aligned}
        \end{equation}
        By \eqref{eq:lem:central.moments.Poisson.examples}, we also have
        \begin{equation}
            \begin{aligned}
                &\big|\EE[(K - \lambda) \, \ind_{\{K\in A\}}]\big| = \big|\EE[(K - \lambda) \, \ind_{\{K\in A^c\}}]\big|, \\
                &\big|\EE[(K - \lambda)^2 \, \ind_{\{K\in A\}}] - \lambda\big| = \big|\EE[(K - \lambda)^2 \, \ind_{\{K\in A^c\}}]\big|, \\
                &\big|\EE[(K - \lambda)^3 \, \ind_{\{K\in A\}}] - \lambda\big| = \big|\EE[(K - \lambda)^3 \, \ind_{\{K\in A^c\}}]\big|.
            \end{aligned}
        \end{equation}
        We get \eqref{eq:cor:central.moments.Poisson.on.events.main} by applying the Cauchy-Schwarz inequality and bounding using \eqref{eq:cor:central.moments.Poisson.on.events.proof.first}.
    \end{proof}

\section{Propagation of errors due to Lemma~2 (iv) of \texorpdfstring{\cite{MR2960952}}{Leblanc (2012a)}}\label{sec:error.propagation}

            The incorrect estimate in Lemma~2 (iv) of \cite{MR2960952}, mentioned in Remark~\ref{rem:incorrect.Leblanc}, has propagated in the literature and caused many subsequent errors in the statements of theorems, propositions and/or lemmas for articles and theses dealing with Bernstein c.d.f.\ estimators.
            Below are 15 articles and theses (in alphabetical order of the authors' last name) where at least one erroneous statement can be traced back to the error in Lemma~2 (iv) of \cite{MR2960952}:

            \begin{itemize}[leftmargin=*]\setlength\itemsep{-1em}
                \item \cite{Belalia2016phd} \vspace{-1mm}
                    \begin{itemize}
                        \item $V_x(y)$ should be equal to $\sqrt{x (1 - x) / (4\pi)} F_x'(y)$ in Proposition~2.2, Theorem~2.2 and Theorem~2.3;
                        \item $V(x,y)$ should be equal to $\big[F_x(x,y) (x (1-x) / \pi)^{1/2} + F_y(x,y) (x (1-x) / \pi)^{1/2}\big]$ in Theorem~4.1, Corollary 4.1, Corollary 4.2 and Theorem~4.2;
                        \item $V_Z(z)$ should be equal to $\big[F_Z(z) (z (1-z) / \pi)^{1/2}\big]$ in Theorem~4.1;
                        \item $V(x_1,\dots,x_d)$ should be equal to $\sum_{j=1}^d \Big\{F_{x_j}(x_1,\dots,x_d) (x_j (1 - x_j) / \pi)^{1/2}\Big\}$ in Remark~ 4.2;
                        \item $\psi_2(x)$ should be equal to $\sqrt{x (1 - x) / (4\pi)}$ in Lemma~4.1 (ii);
                    \end{itemize}
                    \vspace{5mm}
                \item \cite{MR3474765} \vspace{-1mm}
                    \begin{itemize}
                        \item $V(x,y)$ should be equal to $\big[F_x(x,y) (x (1-x) / \pi)^{1/2} + F_y(x,y) (x (1-x) / \pi)^{1/2}\big]$ in Theorem~1, Corollary 1, Corollary 2 and Theorem~2;
                        \item $V_Z(z)$ should be equal to $\big[F_Z(z) (z (1-z) / \pi)^{1/2}\big]$ in Theorem~1;
                        \item $V(x_1,\dots,x_d)$ should be equal to $\sum_{j=1}^d \Big\{F_{x_j}(x_1,\dots,x_d) (x_j (1 - x_j) / \pi)^{1/2}\Big\}$ in Remark~ 2;
                        \item $\psi_2(x)$ should be equal to $\sqrt{x (1 - x) / (4\pi)}$ in Lemma~1 (ii);
                    \end{itemize} 
                    \vspace{5mm}
                \item \cite{MR3630225} \vspace{-1mm}
                    \begin{itemize}
                        \item $V_x(y)$ should be equal to $\sqrt{x (1 - x) / (4\pi)} F_x'(y)$ in Proposition~2, Theorem~2 and Theorem~3;
                    \end{itemize} 
                    \vspace{5mm}
                \item \cite{doi:10.1080/03610926.2020.1734832} \vspace{-1mm}
                    \begin{itemize}
                        \item $V(x,y)$ should be equal to $\big[H_x(x,y) (x (1-x) / \pi)^{1/2} + H_y(x,y) (x (1-x) / \pi)^{1/2}\big]$ in Proposition~3;
                    \end{itemize} 
                    \vspace{5mm}
                \item \cite{MR3473628} \vspace{-1mm}
                    \begin{itemize}
                        \item $V(x_0)$ should be equal to $f(x_0) [x_0 (1 - x_0) / \pi]^{1/2}$ on page 242 and everywhere inside $m_{\text{opt}}$;
                    \end{itemize} 
                    \vspace{5mm}
                \item \cite{MR3899096} \vspace{-1mm}
                    \begin{itemize}
                        \item $V(x)$ should be equal to $f(x) [x (1 - x) / \pi]^{1/2}$ in Theorem~2;
                        \item The r.h.s.\ of (16) should be equal to $W^2(m)^{-1/2} \big\{-[x (1 - x) / (4\pi)]^{1/2} + \OO(1)\big\}$;
                    \end{itemize} 
                    \vspace{5mm}
                \item \cite{Hanebeck2020master} \vspace{-1mm}
                    \begin{itemize}
                        \item $V(x)$ should be equal to $f(x) [x (1 - x) / \pi]^{1/2}$ in Theorem~5.3, Theorem~5.5, Corollary 5.2, Theorem~5.6, Corollary 5.3 and Theorem~5.7;
                        \item $V^S(x)$ should be equal to $f(x) [x / \pi]^{1/2}$ in Theorem~6.4, Theorem~6.6, Corollary 6.2, Theorem~6.7, Corollary 6.3 and Theorem~6.8;
                        \item $R_{1,m}^S(x)$ should be equal to $m^{-1/2} \{-\sqrt{x / (4\pi)} + \oo_x(1)\}$ in Lemma~6.3 (e);
                    \end{itemize}
                    \vspace{5mm}
                \item \cite[arXiv v.1]{arXiv:2005.09994} \vspace{-1mm}
                    \begin{itemize}
                        \item $V^S(x)$ should be equal to $f(x) [x / \pi]^{1/2}$ in Theorem~5, Theorem~7, Corollary 2, Theorem~8, Corollary 3 and Theorem~ 9;
                        \item $R_{1,m}^S(x)$ should be equal to $m^{-1/2} \{-\sqrt{x / (4\pi)} + \oo_x(1)\}$ in Lemma~3 (e);
                    \end{itemize} 
                    {\bf Note: The subsequent arXiv versions of \cite{arXiv:2005.09994} were corrected following Theorem~\ref{thm:asymptotics.variance} and Lemma~\ref{lem:technical.sums.R} of the present paper.}
                    \vspace{5mm}
                \item \cite{Jmaei2018phd} \vspace{-1mm}
                    \begin{itemize}
                        \item $V(x)$ should be equal to $f(x) [x (1 - x) / \pi]^{1/2}$ in Proposition~1.3.2 and in the expression of the MSE and MISE on page 52;
                        \item $V(x)$ should be equal to $f(x) [x (1 - x) / \pi]^{1/2}$ in Proposition~2.3.1, Proposition~2.3.2, Corollary 2.3.1, (2.3.7), (2.3.8) and Remark~2.3.1;
                    \end{itemize}
                    \newpage
                \item \cite{MR3740720} \vspace{-1mm}
                    \begin{itemize}
                        \item $V(x)$ should be equal to $f(x) [x (1 - x) / \pi]^{1/2}$ in Proposition~3.1, Proposition~3.2, Corollary 3.1, (10), (11) and Remark~3.1;
                    \end{itemize} 
                    \vspace{5mm}
                \item \cite{MR2925964} \vspace{-1mm}
                    \begin{itemize}
                        \item $\Psi(x)$ should be equal to $\sqrt{x (1 - x) / \pi}$ in (18) and Lemma~8
                    \end{itemize} 
                    \vspace{5mm}
                \item \cite{Lyu2020master} \vspace{-1mm}
                    \begin{itemize}
                        \item $V(x)$ should be equal to $f(x) [x (1 - x) / \pi]^{1/2}$ in Theorem~3.3;
                        \item $V(x,y)$ should be equal to $\big\{F_x(x,y) (x (1-x) / \pi)^{1/2} + F_y(x,y) (x (1-x) / \pi)^{1/2}\big\}$ in Theorem~3.4;
                        \item $V_Z(z)$ should be equal to $\big\{F_Z(z) (z (1-z) / \pi)^{1/2}\big\}$ in Theorem~3.4;
                        \item $\psi_2(x)$ should be equal to $\sqrt{x (1 - x) / (4\pi)}$ in Lemma~3 (iv);
                    \end{itemize}
                    \vspace{5mm}
                \item \cite{doi:10.1007/s13163-021-00384-0}
                    \begin{itemize}
                        \item $\frac{2 x (1 - x)}{\pi}$ should be replaced by $\frac{x (1 - x)}{\pi}$ everywhere on page 9;
                    \end{itemize}
                    \vspace{5mm}
                \item \cite{TchouakeTchuiguep2013master} \vspace{-1mm}
                    \begin{itemize}
                        \item $V(x)$ should be equal to $f(x) [x (1 - x) / \pi]^{1/2}$ in Th\'eor\`eme 3.4 and Corollaire 3.6;
                        \item $\gamma_2(x)$ should be equal to $\sqrt{x (1 - x) / (4\pi)}$ in Lemma~3.5 (ii);
                    \end{itemize}
                    \vspace{5mm}
                \item \cite{MR3950592} \vspace{-1mm}
                    \begin{itemize}
                        \item $\psi_2(x)$ should be equal to $[t (1 - t) / (4\pi)]^{1/2}$ in Lemma~2 (ii);
                        \item The last term of $V(t)$ in Lemma~3 (iii) should be equal to $\big(t (1 - t) / (4 \pi)\big)^{1/2}$.
                    \end{itemize} 
            \end{itemize}
\end{appendices}

\section*{Acknowledgments}

The author would like to thank an anonymous referee for his valuable comments that led to improvements in the presentation of this paper.

%
%

\section*{References}
\phantomsection
\addcontentsline{toc}{chapter}{References}

\nocite{arXiv:2002.06956}

\bibliographystyle{authordate1}
\bibliography{Ouimet_2021_LLT_Poisson_bib}

\end{document}